\documentclass[12pt,letterpaper,english]{amsart}
\usepackage{charter}
\usepackage{helvet}
\usepackage{courier}
\usepackage[T1]{fontenc}
\usepackage[latin9]{inputenc}
\usepackage{mathrsfs}
\usepackage{amsthm}
\usepackage{amstext}
\usepackage{amssymb}
\usepackage[all]{xy}

\makeatletter

\pdfpageheight\paperheight
\pdfpagewidth\paperwidth

\numberwithin{equation}{section}
\numberwithin{figure}{section}
  \theoremstyle{plain}
  \newtheorem*{thm*}{\protect\theoremname}
  \theoremstyle{plain}
  \newtheorem*{cor*}{\protect\corollaryname}
\theoremstyle{plain}
\newtheorem{thm}{\protect\theoremname}[section]
  \theoremstyle{definition}
  \newtheorem{defn}[thm]{\protect\definitionname}
  \theoremstyle{remark}
  \newtheorem{rem}[thm]{\protect\remarkname}
  \theoremstyle{remark}
  \newtheorem*{rem*}{\protect\remarkname}
  \theoremstyle{plain}
  \newtheorem{prop}[thm]{\protect\propositionname}
  \theoremstyle{plain}
  \newtheorem{cor}[thm]{\protect\corollaryname}

\usepackage{verbatim} 
\usepackage{eucal}
\usepackage{amsthm}
\usepackage[all]{xy}
\usepackage{manfnt}
\usepackage{float}
\usepackage{xr}
\SelectTips{cm}{}
\xyoption{line}
\usepackage{hyperref}
\makeatletter \newcommand{\xyR}[1]{%
\makeatletter \xydef@\xymatrixrowsep@{#1} \makeatother }
\makeatletter \newcommand{\xyC}[1]{%
\makeatletter \xydef@\xymatrixcolsep@{#1} \makeatother }
\mathcode`\:="603A 
\DeclareSymbolFont{rsfs}{U}{rsfs}{m}{n}
\DeclareSymbolFontAlphabet{\mathrf}{rsfs}

\makeatother

\usepackage{babel}
  \providecommand{\corollaryname}{Corollary}
  \providecommand{\definitionname}{Definition}
  \providecommand{\propositionname}{Proposition}
  \providecommand{\remarkname}{Remark}
  \providecommand{\theoremname}{Theorem}
\providecommand{\theoremname}{Theorem}

\begin{document}

\title{Steenrod coalgebras of simplicial complexes}

\author{Justin R. Smith}



\curraddr{Department of Mathematics\\
Drexel University\\
Philadelphia,~PA 19104}

\email{jsmith@drexel.edu}

\urladdr{http://vorpal.math.drexel.edu}

\maketitle
\global\long\def\ring{\mathbb{Z}}
\global\long\def\integers{\mathbb{Z}}
\global\long\def\betabar{\bar{\beta}}
 \global\long\def\desusp{\downarrow}
\global\long\def\susp{\uparrow}
\global\long\def\cobar{\mathcal{F}}
\global\long\def\coend{\mathrm{CoEnd}}
\global\long\def\ainfty{A_{\infty}}
\global\long\def\coassoc{\mathrm{Coassoc}}
\global\long\def\trm{\mathrm{T}}
\global\long\def\tfr{\mathfrak{T}}
\global\long\def\tabbr{\hat{\trm}}
\global\long\def\Tabbr{\hat{\tfr}}
\global\long\def\afr{\mathfrak{A}}
\global\long\def\homz{\mathrm{Hom}_{\ring}}
\global\long\def\zend{\mathrm{End}}
\global\long\def\rs#1{\mathrm{R}S_{#1 }}
\global\long\def\forgetful#1{\lceil#1\rceil}
\global\long\def\highprod#1{\bar{\mu}_{#1 }}
\global\long\def\slength#1{|#1 |}
\global\long\def\barcs{\bar{\mathcal{B}}}
\global\long\def\ubarcs{\mathcal{B}}
\global\long\def\zs#1{\ring S_{#1 }}
\global\long\def\homzs#1{\mathrm{Hom}_{\ring S_{#1 }}}
\global\long\def\zpi{\mathbb{Z}\pi}
\global\long\def\D{\mathfrak{D}}
\global\long\def\ahat{\hat{\mathfrak{A}}}
\global\long\def\cbar{{\bar{C}}}
\global\long\def\cf#1{C(#1 )}
\global\long\def\ddelta{\dot{\Delta}}
\global\long\def\dimlimiter{\triangleright}
\global\long\def\coalgcat{\mathrf S_{0}}
\global\long\def\hcoalgcat{\mathrf{S}}
\global\long\def\ircoalgcat{\mathrf I_{0}}
\global\long\def\bircoalgcat{\mathrf{I}_{0}^{+}}
\global\long\def\hircoalgcat{\mathrf I}
\global\long\def\dcoalgcat{\mathrm{ind}-\coalgcat}
\global\long\def\chaincat{\mathbf{Ch}}
\global\long\def\coll{\mathrm{Coll}}
\global\long\def\bchaincat{\mathbf{Ch}_{0}}
\global\long\def\ilimit{\varprojlim\,}
\global\long\def\bigboxtimes{\mathop{\boxtimes}}
\global\long\def\dlimit{\varinjlim\,}
\global\long\def\coker{\mathrm{{coker}}}
\global\long\def\icoalgcat{\mathrm{pro}-\mathrf{S}_{0}}
\global\long\def\iircoalgcat{\mathrm{pro-}\ircoalgcat}
\global\long\def\dircoalgcat{\mathrm{ind-}\ircoalgcat}
\global\long\def\core#1{\left\langle #1\right\rangle }
\global\long\def\ilimitder{\varprojlim^{1}\,}
\global\long\def\pcoalg#1#2{P_{\mathcal{#1}}(#2) }
\global\long\def\pcoalgf#1#2{P_{\mathcal{#1}}(\forgetful{#2}) }
\global\long\def\coequalizer{\mathop{\mathrm{coequalizer}}}

\global\long\def\mainoperad{\mathcal{H}}
\global\long\def\cone#1{\mathrm{Cone}(#1)}

\global\long\def\im{\operatorname{im}}

\global\long\def\lcell{L_{\mathrm{cell}}}
\global\long\def\ccoalgcat{\mathrf S_{\mathrm{cell}}}

\global\long\def\fc#1{\mathrm{hom}(\bigstar,#1)}
\global\long\def\coS{\mathbf{coS}}
\global\long\def\cocell{\mathbf{co}\ccoalgcat}

\global\long\def\ccoalgcat{\mathrf S_{\mathrm{cell}}}

\global\long\def\spaces{\mathbf{SS}}

\global\long\def\pgam{\tilde{\Gamma}}
\global\long\def\pz{\tilde{\integers}}

\global\long\def\moore#1{\{#1\}}

\global\long\def\ints{\mathbb{Z}}

\global\long\def\finite{\mathcal{F}}

\global\long\def\finiteop{\finite^{\mathrm{op}}}

\global\long\def\syms{\mathbf{SS}}

\global\long\def\ordered{\mathbf{\Delta}}

\global\long\def\sets{\mathbf{Set}}

\global\long\def\colim{\operatorname{colim}}

\newdir{ >}{{}*!/-5pt/@{>}}

\global\long\def\treal#1{\mathcal{T}(\bigstar,#1)}

\global\long\def\rats{\mathbb{Q}}

\global\long\def\img{\operatorname{im}}

\global\long\def\tmap#1{\mathrm{T}_{#1}}

\global\long\def\Tmap#1{\mathfrak{T}_{#1}}

\global\long\def\glist#1#2#3{#1_{#2},\dots,#1_{#3}}

\global\long\def\blist#1#2{\glist{#1}1{#2}}

\global\long\def\enlist#1#2{\{\blist{#1}{#2}\}}

\global\long\def\tlist#1#2{\tmap{\blist{#1}{#2}}}

\global\long\def\Tlist#1#2{\Tmap{\blist{#1}{#2}}}

\global\long\def\nth#1{\mbox{#1}^{\mathrm{th}}}

\global\long\def\tunder#1#2{\tmap{\underbrace{{\scriptstyle #1}}_{#2}}}

\global\long\def\Tunder#1#2{\Tmap{\underbrace{{\scriptstyle #1}}_{#2}}}

\global\long\def\tunderi#1#2{\tunder{1,\dots,#1,\dots,1}{#2^{\mathrm{th}}\ \mathrm{position}}}

\global\long\def\Tunderi#1#2{\Tunder{1,\dots,#1,\dots,1}{#2^{\mathrm{th}}\,\mathrm{position}}}

\global\long\def\chaincat{\mathbf{Ch}}

\global\long\def\chaincatp{\chaincat_{0}}

\global\long\def\simpc{\mathbf{SC}}

\global\long\def\s{\mathfrak{S}}

\global\long\def\pco{P_{\s}}

\global\long\def\lco{L_{\s}}

\global\long\def\ns#1{\mathcal{N}^{#1}}

\global\long\def\cfn#1{N(#1)}

\global\long\def\freeop{\mathcal{F}}

\global\long\def\kerftos{\mathscr{R}}

\global\long\def\comm{\mathbf{Commute}}

\global\long\def\steen{\mathscr{S}}

\global\long\def\arity{\operatorname{arity}\,}

\global\long\def\nfc#1{\mathrm{hom}_{\steen}(\bigstar,#1)}

\global\long\def\dcat{\mathbf{D}}

\global\long\def\ords{\mathbf{\Delta}_{+}}

\global\long\def\sd{\mathfrak{f}}

\global\long\def\ds{\mathfrak{d}}

\global\long\def\ss{\mathbf{S}}

\date{\today}
\begin{abstract}
In this paper, we extend earlier work by showing that if $X$ and
$Y$ are ordered simplicial complexes (i.e. simplicial sets whose
simplices are determined by their vertices), a\emph{ }morphism $g:\cfn X\to\cfn Y$
of Steenrod coalgebras (normalized chain-complexes equipped with extra
structure) induces one of topological realizations $\hat{g}:|X|\to|Y|$.
If $g$ is an isomorphism, then it induces an isomorphism between
$X$ is and $Y$, implying that $|X|$ and $|Y|$ are homeomorphic.
\end{abstract}

\section{Introduction}

It is well-known that the Alexander-Whitney coproduct is functorial
with respect to simplicial maps. If $X$ is a simplicial set, $C(X)$
is the unnormalized chain-complex and $\rs 2$ is the \emph{bar-resolution}
of $\ints_{2}$ (see \cite{maclane:1975}), it is also well-known
that there is a unique homotopy class of $\ints_{2}$-equivariant
maps (where $\ints_{2}$ transposes the factors of the target) 
\[
\xi_{X}:\rs 2\otimes C(X)\to C(X)\otimes C(X)
\]
cohomology, and that this extends the Alexander-Whitney diagonal.
We will call such structures, Steenrod coalgebras and the map $\xi_{X}$
the Steenrod diagonal. Done carefully (see appendix B in \cite{smith-steenrod}),
this Steenrod diagonal is functorial.

In \cite{smith:1994}, the author defined the functor $\mathcal{C}(*)$
on simplicial sets --- essentially the chain complex equipped with
the structure of a coalgebra over an operad $\s$. This coalgebra
structure determined all Steenrod and other cohomology operations.
Since these coalgebras are not \emph{nilpotent}\footnote{In a nilpotent coalgebra, iterated coproducts of elements ``peter
out'' after a finite number of steps. See \cite[chapter~3]{operad-book}
for the precise definition.}\emph{ }they have a kind of ``transcendental'' structure that contains
much more information. In \cite{smith-steenrod}, the author showed
that this transcendental structure even manifests in the sub-operad
of $\s$ generated by $\s(2)=\rs 2$ and proved
\begin{thm*}
If $X$ and $Y$ are pointed reduced simplicial sets and 
\[
f:C(X)\to C(Y)
\]
is a morphism of Steenrod coalgebras --- over unnormalized chain-complexes
--- then $f$ induces a commutative diagram
\[
\xymatrix{{X} & {Y}\\
\ds\circ\sd(X)\ar[d]_{\phi_{(\ds\circ\sd(X))}}\ar[u]^{g_{X}} & \ds\circ\sd(Y)\ar[d]^{\phi_{(\ds\circ\sd(Y))}}\ar[u]_{g_{Y}}\\
{\ints_{\infty}(\ds\circ\sd(X))}\ar[r]^{f_{\infty}}\ar[d]_{q_{(\ds\circ\sd(X))}} & {\ints_{\infty}(\ds\circ\sd(Y))}\ar[d]^{q_{(\ds\circ\sd(Y))}}\\
{\pz(\ds\circ\sd(X))}\ar[r]_{\pgam f} & {\pz(\ds\circ\sd(Y))}
}
\]
 where $g_{X}$ and $g_{Y}$ are homotopy equivalences if $X$ and
$Y$ are Kan complexes --- and homotopy equivalences of their topological
realizations otherwise. In particular, if $X$ and $Y$ are nilpotent
and $f$ is an integral homology equivalence, then the topological
realizations $|X|$ and $|Y|$ are homotopy equivalent.
\end{thm*}
Here, $\sd$ and $\ds$ are functors defined in definition~\ref{def:sd-ds-functors}.

It follows that that the $\cf *$-functor determines a \emph{nilpotent
}space's weak homotopy type. In the present paper, we complement the
results of \cite{smith-steenrod} by showing:
\begin{cor*}
\ref{cor:cellular-determines-pi1}. If $X$ and $Y$ are ordered simplicial
complexes, any purely algebraic chain map of normalized chain complexes\emph{
}
\[
f:\cfn X\to\cfn Y
\]
 that makes the diagram 
\begin{equation}
\xymatrix{{\rs 2\otimes\cfn X}\ar[r]^{1\otimes f}\ar[d]_{\xi_{X}} & {\rs 2\otimes\cfn Y}\ar[d]^{\xi_{Y}}\\
{\cfn X\otimes\cfn X}\ar[r]_{f\otimes f} & {\cfn Y\otimes\cfn Y}
}
\label{eq:coproduct-diagram-1}
\end{equation}
commute induces a map of simplicial sets\emph{
\[
\hat{f}:\ds(X)\to\ds(Y)
\]
}which are equipped with canonical inclusions
\begin{align*}
\iota_{X}:X & \to\ds(X)\\
\iota_{Y}:Y & \to\ds(Y)
\end{align*}
that induce homeomorphisms of their topological realizations. If $f$
is an isomorphism, then X and $Y$ are isomorphic, hence homeomorphic.

In all cases, the diagram
\[
\xymatrix{{H_{i}(N(X))}\ar[r]^{g}\ar[d]_{H_{i}(j_{X})}^{\cong} & {H_{i}(N(Y))}\ar[d]_{\cong}^{H_{i}(j_{Y})}\\
{H_{i}(C(\ds(X)))}\ar[r]_{H_{*}(\hat{g})} & {H_{i}(C(\ds(Y)))}
}
\]
commutes for all $i\ge0$, where $j_{X}$ and $j_{Y}$ are chain-maps
induced by the inclusion of $N(X)$ and $N(Y)$ into the $C(\ds(X))$
and $C(\ds(Y))$, respectively.
\end{cor*}
Recall that an \emph{ordered simplicial complex} is a simplicial set
without degeneracies whose simplices are uniquely determined by their
vertices (for instance, a piecewise linear manifold). The proof \emph{requires}
$X$ and $Y$ to be ordered simplicial complexes and is likely not
true for arbitrary simplicial sets. Also note that we require diagram~\ref{eq:coproduct-diagram-1}
to commute \emph{exactly,} not merely up to a chain-homotopy (as is
done when using it to compute Steenrod squares).

This and the main result in \cite{smith-steenrod} imply that old
mathematical structures like chain-complexes and Steenrod diagonals
encapsulate \emph{vast} amounts of information about a space --- and
that the traditional ways of studying them (taking cohomology, for
example) throw most of this information away.

The author is indebted to Dennis Sullivan for several interesting
discussions.

\section{Definitions and assumptions}

Throughout this paper $C(*)$ will denote the unnormalized chain complex
and $N(*)$ the normalized one.

If $C$ is a chain-complex
\begin{equation}
C^{\otimes n}=\underbrace{C\otimes\cdots\otimes C}_{n\text{ factors}}\label{eq:otimesn}
\end{equation}

\begin{defn}
\label{def:chaincat} We will denote the category of $\ring$-free
chain chain-complexes by $\chaincat$ and ones that are\emph{ bounded
from below} in dimension $0$ by $\bchaincat$.
\end{defn}
We make extensive use of the Koszul Convention (see~\cite{gugenheim:1960})
regarding signs in homological calculations:
\begin{defn}
\label{def:koszul} If $f:C_{1}\to D_{1}$, $g:C_{2}\to D_{2}$ are
maps, and $a\otimes b\in C_{1}\otimes C_{2}$ (where $a$ is a homogeneous
element), then $(f\otimes g)(a\otimes b)$ is defined to be $(-1)^{\deg(g)\cdot\deg(a)}f(a)\otimes g(b)$. \end{defn}
\begin{rem}
If $f_{i}$, $g_{i}$ are maps, it isn't hard to verify that the Koszul
convention implies that $(f_{1}\otimes g_{1})\circ(f_{2}\otimes g_{2})=(-1)^{\deg(f_{2})\cdot\deg(g_{1})}(f_{1}\circ f_{2}\otimes g_{1}\circ g_{2})$.
\end{rem}
The set of morphisms of chain-complexes is itself a chain complex:
\begin{defn}
\label{def:homcomplex}Given chain-complexes $A,B\in\chaincat$ define
\[
\homz(A,B)
\]
to be the chain-complex of graded $\ring$-morphisms where the degree
of an element $x\in\homz(A,B)$ is its degree as a map and with differential
\[
\partial f=f\circ\partial_{A}-(-1)^{\deg f}\partial_{B}\circ f
\]
As a $\ring$-module $\homz(A,B)_{k}=\prod_{j}\homz(A_{j},B_{j+k})$.\end{defn}
\begin{rem*}
Given $A,B\in\chaincat^{S_{n}}$, we can define $\homzs n(A,B)$ in
a corresponding way.
\end{rem*}

\section{Steenrod coalgebras\label{sec:Steenrod-coalgebras}}

We begin with:
\begin{defn}
\label{def:Steenrod-coalgebra}A \emph{Steenrod coalgebra,} $(C,\delta)$
is a chain-complex $C\in\chaincat$ equipped with a $\ints_{2}$-equivariant
chain-map
\[
\delta:\rs 2\otimes C\to C\otimes C
\]
 where $\ints_{2}$ acts on $C\otimes C$ by swapping factors and
$\rs 2$ is the bar-resolution of $\ints$ over $\zs 2$. A morphism
$f:(C,\delta_{C})\to(D,\delta_{D})$ is a chain-map $f:C\to D$ that
makes the diagram
\[
\xyR{30pt}\xymatrix{{\rs 2\otimes C}\ar[r]^{1\otimes f}\ar[d]_{\delta_{C}} & {\rs 2\otimes D}\ar[d]^{\delta_{D}}\\
{C\otimes C}\ar[r]_{f\otimes f} & {D\otimes D}
}
\]
commute.
\end{defn}
Steenrod coalgebras are very general --- the underlying coalgebra
need not even be coassociative. The category of Steenrod coalgebras
is denoted $\steen$.
\begin{rem*}
It turns out that Steenrod coalgebras are coalgebras over the free
operad generated by $\rs 2$. We will not need this fact in this paper.\end{rem*}
\begin{defn}
\label{def:hndef}If 
\[
\delta:\rs 2\otimes C\to C\otimes C
\]
 is a Steenrod coalgebra, the structure map above induces the adjoint
structure map
\begin{equation}
\alpha:C\to\homzs 2(\rs 2,C\otimes C)\subset\homz(\rs 2,C\otimes C)\label{eq:adjoint-struct-map}
\end{equation}
Let
\[
H_{2}(C)=\homz(\rs 2,C\otimes C)
\]
 and inductively define 
\[
H_{n}(C)=\homz(\rs 2,H_{n-1}(C)\otimes C)
\]
with chain-maps
\begin{align}
\alpha_{2}=\alpha:C & \to H_{2}(C)\nonumber \\
\alpha_{n}=\homz(1,\alpha_{n-1}\otimes1)\circ\alpha:C & \to H_{n}(C)\label{eq:induction1}
\end{align}
for all $n\ge2$.\end{defn}
\begin{prop}
\label{prop:infinfinite-iteration}Under the hypotheses of definition\ref{def:hndef},
there exist chain-maps
\[
\beta_{n}:H_{n}(C)\to\homz(\rs 2^{\otimes(n-1)},C^{\otimes n})
\]
for all $n\ge2$. It follows that the adjoint structure map induces
a chain-map
\[
A:C\to\prod_{n=2}^{\infty}\homz(\rs 2^{\otimes(n-1)},C^{\otimes n})
\]
\end{prop}
\begin{rem*}
See equation~\ref{eq:otimesn} for the notation $C^{\otimes n}$.\end{rem*}
\begin{proof}
The map $\beta_{2}$ is the identity. For larger values of $n$, the
existence of $\beta_{n}$ is inductively defined as the composite
\begin{multline}
\homz(\rs 2,H_{n-1}(C)\otimes C)\xrightarrow{\homz(1,\beta_{n-1}\otimes1)}\\
\homz(\rs 2,\homz(\rs 2^{\otimes(n-2)},C^{\otimes(n-1)})\otimes C)\\
\xrightarrow{\ell}\homz(\rs 2^{\otimes(n-1)},C^{\otimes n})\label{eq:induction2}
\end{multline}

\end{proof}

\section{Simplices and complexes\label{sec:morphisms}}

In this section, we consider properties of Steenrod coalgebras that
are topologically derived from simplices and simplicial complexes
via the construction in appendix~B of \cite{smith-steenrod}.

The key result is proposition~B.5 of \cite{smith-steenrod}, which
proves that if $e_{n}=\underbrace{[(1,2)|\cdots|(1,2)]}_{n\text{ terms}}\in\rs 2$
and $x\in\cfn X_{k}$ is the image of a $k$-simplex, then
\[
\xi_{X}(e_{k}\otimes x)=\eta_{k}\cdot x\otimes x
\]
where $\eta_{k}=(-1)^{k(k+1)/2}$ and 
\[
\xi_{X}:\rs 2\otimes\cfn X\to\cfn X\otimes\cfn X
\]
is the Steenrod diagonal (see definition~\ref{def:Steenrod-coalgebra}).
\begin{defn}
\label{def:gamma-m-map}If $k,m$ are positive integers, $C$ is a
chain-complex, and $E_{2,m}=e_{m}$ and $E_{k,m}=\underbrace{e_{m}\otimes\cdots\otimes e_{m}}_{k-1\text{ iterations}}\in\rs 2^{\otimes(k-1)}$
\[
\rho_{m}=(\eta_{m}\cdot E_{2,m},\eta_{m}^{2}\cdot E_{3,m},\eta_{m}^{3}\cdot E_{4,m},\dots)\in\prod_{n=2}^{\infty}\rs 2^{\otimes(n-1)}
\]
with $\eta_{m}=(-1)^{m(m+1)/2}$ (see proposition~B.5 of \cite{smith-steenrod})
and define 
\[
\gamma_{m}:\prod_{n=2}^{\infty}\homz(\rs 2^{\otimes(n-1)},C_{m}^{\otimes n})\to\prod_{n=2}^{\infty}C^{\otimes n}
\]
via evaluation on $\rho_{m}$.
\end{defn}
We have
\begin{cor}
\label{cor:simplex-image}If $X$ is an ordered simplicial complex
and $c\in N(X)_{n}$ is an element generated by an $n$-simplex, then
the image of $c$ under the composite, $\Xi_{n}$,
\[
\cfn X_{n}\xrightarrow{A}\prod_{k=2}^{\infty}\homz(\rs 2^{\otimes(k-1)},C_{n}^{\otimes k})\xrightarrow{\gamma_{n}}\prod_{k=2}^{\infty}\cfn X^{\otimes k}
\]
 is 
\begin{equation}
\Xi_{n}(c)=(c,c\otimes c,\dots)\label{eq:xibigeq}
\end{equation}

Here, $N(X)$ is the (normalized) chain complex of $X$ and the chain-map,
$A$, is defined in proposition \ref{prop:infinfinite-iteration}.\end{cor}
\begin{rem*}
Since $\Xi_{n}$ is constructed using the Steenrod coalgebra structure
of $N(X)$, it is natural with respect to Steenrod coalgebra morphisms.
Equation~\ref{eq:xibigeq} is generally only valid for chain-complexes
of \emph{simplicial sets} and elements, $c$, induced by \emph{actual
simplices.}\end{rem*}
\begin{proof}
We claim that 
\[
\beta_{j}(c)=(\eta_{n}^{j-1}\cdot E_{j,n}\mapsto c^{\otimes j})\in\homz(\rs 2^{\otimes(j-1)},N(X)_{j}^{\otimes j})
\]
where we follow the notation of proposition~\ref{prop:infinfinite-iteration}.
When $j=2$, this follows from proposition~B.5 of \cite{smith-steenrod}
and the fact that $c$ is the image of an $n$-simplex. If $j>2$
it follows from the case $j=2$ and induction on $j$:
\begin{align*}
\beta_{j}(c) & =\ell\circ\homz(1,\beta_{j-1}\otimes1)(\eta_{n}\cdot E_{2,n}\mapsto c\otimes c) & \text{by equation }\ref{eq:induction2}\\
 & =\ell(\eta_{n}\cdot E_{2,n}\mapsto\beta_{j-1}(c)\otimes c)\\
 & =\ell(\eta_{n}\cdot E_{2,n}\mapsto(\eta_{n}^{j-2}\cdot E_{j-1,n}\mapsto c^{\otimes(j-1)})\otimes c) & \text{by induction}\\
 & =(\eta_{n}\cdot E_{2,n}\otimes\eta_{n}^{j-2}\cdot E_{j-1,n}\mapsto c^{\otimes j})\\
 & =(\eta_{n}^{j-1}\cdot E_{j,n}\mapsto c^{\otimes j}) & \text{definition }\ref{def:gamma-m-map}
\end{align*}

If 
\[
p_{j}:\prod_{k=2}^{\infty}\cfn X^{\otimes k}\to\cfn X^{\otimes j}
\]
is the projection, then $p_{j}(\Xi_{n}(c))=c^{\otimes j}$ and the
conclusion follows.
\end{proof}
Lemma~C.1 of \cite{smith-steenrod} implies that:
\begin{cor}
\label{cor:n-simplices-map-to-simplices}Let $X$ be a simplicial
set and suppose 
\[
f:\ns n=\cfn{\Delta^{n}}\to\cfn X
\]
 is a Steenrod coalgebra morphism. Then the image of the generator
$\Delta^{n}\in\cfn{\Delta^{n}}_{n}$ is zero or a generator of $\cfn X_{n}$
defined by an $n$-simplex of $X$.\end{cor}
\begin{rem*}
As the statement implies, we do not need $X$ to be an ordered simplicial
complex in this result.\end{rem*}
\begin{proof}
Since $f$ is a Steenrod coalgebra morphism, it induces a commutative
diagram
\[
\xyC{50pt}\xymatrix{{N(\Delta^{n})}\ar[r]^{\Xi_{n,N(\Delta^{n})}\qquad}\ar[d]_{f} & {\prod_{j=2}^{\infty}N(\Delta^{n})^{\otimes j}}\ar[d]^{\prod_{j=2}^{\infty}f^{\otimes j}}\\
{N(X)}\ar[r]_{\Xi_{n,N(X)}\qquad} & {\prod_{j=2}^{\infty}N(X)^{\otimes j}}
}
\]

Suppose 
\[
f(\Delta^{n})=\sum_{k=1}^{t}c_{k}\cdot\sigma_{k}^{n}\in\cfn X
\]
where the $\sigma_{k}^{n}$ are images of $n$-simplices of $X$ and
the $c_{k}\in\ints$. If $f(\Delta^{n})$ is not \emph{equal} to one
of the $\sigma_{k}^{n}$ (i.e. if more than one of the $c_{k}$ is
nonzero, or if only one is nonzero but not equal to $+1$), lemma~C.1
of \cite{smith-steenrod} implies that its image under $\Xi_{N(X)}$
in corollary~\ref{cor:simplex-image} is \emph{linearly independent
}of the images of the $\sigma_{k}^{n}$, a contradiction. 
\end{proof}
We also conclude that:
\begin{cor}
\label{cor:automorphisms-trivial}If $f:\cfn{\Delta^{n}}\to\cfn{\Delta^{n}}$
is 
\begin{enumerate}
\item an isomorphism of Steenrod coalgebras in dimension $n$ and 
\item an endomorphism of Steenrod coalgebras in lower dimensions 
\end{enumerate}
then $f$ is the identity map.\end{cor}
\begin{proof}
Corollary~\ref{cor:n-simplices-map-to-simplices} implies that $f$
maps every sub-simplex of $\Delta^{n}$ to one of the same dimension.
We may identify a $k$ dimensional sub-simplex of $\Delta^{n}$ with
a set of $k+1$ vertices $\{i_{0},\dots,i_{k}\}$ with $i_{0}<\cdots<i_{k}$.

We are given that $f$ is an isomorphism in dimension $n$ --- i.e.
it is bijective. We use downward induction on dimension to show that
it is bijective in lower dimensions.

If $f$ is bijective in dimension $k$, \emph{every} set of $k+1$
vertices $\{j_{0},\dots,i_{k}\}$ occurs exactly \emph{once} as $f(\Delta_{\ell}^{k})$
for some $\ell$. Given any $k$-simplex, $\Delta^{k}$, with $f(\Delta^{k})=\{j_{0},\dots,i_{k}\}$,
the boundary $\partial f(\Delta^{k})$ is a linear combination of
$k+1$ \emph{distinct} faces --- namely all $k$-element subsets of
$\{j_{0},\dots,i_{k}\}$. Since $f$ is a chain-map, $f(\partial\Delta^{k})$
must be a linear combination of all $k$-element subsets of $\{j_{0},\dots,j_{k}\}$.
It follows that \emph{every} $k$-element subset of \emph{every} $k+1$-element
set occurs in $f(\Delta_{t}^{k-1})$ for \emph{some} $t=1,\dots,{n+1 \choose k}$.
The Pigeonhole Principle implies that each such $k$-element subset
occurs exactly \emph{once }in the image of \emph{$f$, }so that $f$
is bijective on $k-1$-simplices.

We conclude that $f$ is an \emph{automorphism} of $\cfn{\Delta^{n}}$.
Now we show that $f$ is the \emph{identity map:}

In dimension $0$, let $f$ be a permutation, $\pi:\{0,\dots,n\}\to\{0,\dots,n\}$
of vertices. If $s=(i_{1},i_{2})$ with $i_{1}<i_{2}$ is a 1-simplex,
$f(s)=(j_{1},j_{2})$ with $j_{1}<j_{2}$ is a 1-simplex, and 
\[
f(\partial s)=f(i_{1})-f(i_{2})=\partial f(s)=(j_{1})-(j_{2})=(\pi i_{1})-(\pi i_{2})
\]
Given the \emph{signs} of the terms in the boundary, we conclude that
$i_{1}<i_{2}\implies\pi i_{1}<\pi i_{2}$ for all $0\le i_{1}<i_{2}\le n$
(in other words, $\pi$ cannot \emph{swap} the ends of a 1-simplex).
This forces $\pi$ to be the identity permutation. It follows that
$f$ is the identity map on 1-simplices.

If $k>1$, $w=(i_{0},\dots,i_{k})$ is any $k$-simplex in $\Delta^{n}$,
and 
\[
\delta_{k}=\underbrace{(1\otimes\cdots\otimes\delta)\circ\cdots\circ\delta}_{k-1\text{ factors}}:\cfn{\Delta^{n}}\to\cfn{\Delta^{n}}^{\otimes k}
\]
where $\delta:\cfn{\Delta^{n}}\to\cfn{\Delta^{n}}\otimes\cfn{\Delta^{n}}$
is the Alexander-Whitney diagonal, then the image of $\delta_{k}(w)$
in 
\[
\cfn{\Delta^{n}}^{\otimes k}/\left(\cfn{\Delta^{n}}^{\otimes k}\right)_{0}
\]
is
\[
Z=(i_{0},i_{1})\otimes(i_{1},i_{2})\otimes\cdots\otimes(i_{k-1},i_{k})\in\cfn{\Delta^{n}}_{1}^{\otimes k}
\]
 where each edge, $(i_{t},i_{t+1})$, is the result of a sequence,
$F_{0}\cdots F_{t-1}F_{t+1}\cdots F_{n}$, of face-operations applied
to $w$. Since these edges are mapped via the identity map (by the
argument above) $f^{\otimes k}(Z)=Z\in\cfn{\Delta^{n}}_{1}^{\otimes k}$,
which implies that $f(w)$ has the same vertices as $w$ so $f(w)=w$.
It follows that $f$ is the identity map in all dimensions.
\end{proof}
A similar line of reasoning implies that:
\begin{cor}
\label{cor:cf-gives-simplices}Let $X$ be an ordered simplicial complex
and let 
\[
f:\cfn{\Delta^{n}}\to\cfn X
\]
map $\Delta^{n}$ to an $n$-simplex $\sigma\in N(X)$ defined by
the inclusion $\iota:\Delta^{n}\to X$. Then
\[
f(\cfn{\Delta^{n}})\subset\cfn{\iota}(\cfn{\Delta^{n}})
\]
so that $f=\cfn{\iota}$.\end{cor}
\begin{proof}
Since $X$ is an ordered simplicial complex, the map $\iota$ is an
inclusion.

Suppose $\Delta^{k}\subset\Delta^{n}$ and $f(\cfn{\Delta^{k}})_{k}\subset\cfn{\Delta^{k}}_{k}$.
Since the boundary of $\Delta^{k}$ is an alternating sum of $k+1$
faces, and since they must map to $k-1$-dimensional simplices of
$\cfn{f(\Delta^{k})}$ with the same signs (so no cancellations can
take place) we must have $f(F_{i}\Delta^{k})\subset\cfn{f(\Delta^{k})}$
and the conclusion follows by downward induction on dimension. The
final statements follow immediately from corollary~\ref{cor:automorphisms-trivial}.
\end{proof}
Next, we consider \emph{degeneracies:}
\begin{prop}
\label{prop:degeneracies}If $n>m$, then surjective Steenrod-coalgebra
morphisms
\[
f:N(\Delta^{n})\to N(\Delta^{m})
\]
 are in a 1-1 correspondence with surjective morphisms
\[
\mathbf{f}:\mathbf{n}\twoheadrightarrow\mathbf{m}
\]
of ordered sets, where $\mathbf{n}=0<\cdots<n$ and $\mathbf{m}=0<\cdots<m$.
In particular, $f$ is induced by the simplicial map
\[
\hat{f}:\Delta^{n}\to\Delta^{m}
\]
corresponding to $\mathbf{f}$.\end{prop}
\begin{proof}
Certainly any surjective Steenrod-coalgebra morphism, $f$, defines
a surjective morphism of vertices: $\mathbf{f}=\alpha(f)$. Given
$\mathbf{f}$, corollary~\ref{cor:n-simplices-map-to-simplices}
implies that the $m$-dimensional sub-simplices of $\Delta^{n}$ can
either map to $\Delta^{m}$ (in a \emph{unique} way, by corollary~\ref{cor:automorphisms-trivial})
or $0$. The sets
\[
\mathbf{f}^{-1}(0),\dots,\mathbf{f}^{-1}(m)
\]
represent sub-simplices of $\Delta^{n}$, which we can imagine that
$\mathbf{f}$ collapses to points --- defining a morphism of ordered
simplicial complexes and a chain-map. Each possible selection $i_{0}\in\mathbf{f}^{-1}(0),\dots,i_{m}\in\mathbf{f}^{-1}(m)$
defines a unique $m$-simplex $\Delta_{i_{0},\dots,i_{m}}^{m}\subset\Delta^{n}$
for which there is a \emph{unique} Steenrod coalgebra morphism (by
corollary~\ref{cor:automorphisms-trivial}) 
\begin{equation}
f_{i_{0},\dots,i_{m}}:N(\Delta_{i_{0},\dots,i_{m}}^{m})\to N(\Delta^{m})\label{eq:degeneracy-morphisms}
\end{equation}
We can define a Steenrod coalgebra morphism 
\[
f:N(\Delta^{n})\to N(\Delta^{m})
\]
 that sends \emph{each} of these to $N(\Delta^{m})$ and all other
sub-simplices of $\Delta^{n}$ to $0$. We will call this morphism
$\beta(\mathbf{f})$.

It is not hard to see that $\mathbf{f}=\alpha\circ\beta(\mathbf{f})$.
That $f=\beta\circ\alpha(f)$ follows from the \emph{uniqueness} of
the morphisms $\{f_{i_{0},\dots,i_{m}}\}$ in equations~\ref{eq:degeneracy-morphisms}.
It follows that $\alpha$ and $\beta$ define inverse one-to-one correspondences.

The final statement follows from the \emph{uniqueness} of Steenrod-coalgebra
morphisms corresponding to a surjective morphism $\mathbf{f}:\mathbf{n}\twoheadrightarrow\mathbf{m}$
and the fact that a simplicial map
\[
\hat{f}:\Delta^{n}\to\Delta^{m}
\]
induces a Steenrod-coalgebra morphism.
\end{proof}
We finally have:
\begin{prop}
\label{prop:image-steenrod-is-simplex}Let $X$ be an ordered simplicial
complex whose vertices form an ordered set $\mathbf{m}$ (which induces
orderings on all of the simplices of $X$). In addition, let $f:N(\Delta^{n})\to N(X)$
be a Steenrod coalgebra morphism whose restriction to vertices induces
a set-map
\[
\mathbf{f}:\mathbf{n}\to\mathbf{m}
\]
with $\mathbf{f}(\mathbf{n})=\{i_{0},\dots,i_{k}\}\subset\mathbf{m}$.
Then
\begin{enumerate}
\item the set-map, $\mathbf{f}$, preserves the order-relation, and
\item if $\{j_{0},\dots,j_{k}\}\subset\mathbf{n}$ is any $k+1$-element
set with $\mathbf{f}(\{j_{0},\dots,j_{k}\})=\{i_{0},\dots,i_{k}\}$,
then $f(N(\Delta_{j_{0},\dots,j_{k}}^{k}))=N(\Delta_{i_{0},\dots,i_{k}}^{k})\subset N(X)$
--- where $\Delta_{j_{0},\dots,j_{k}}^{k}\subset\Delta^{n}$ is the
$k$-dimensional sub-simplex defined by $\{j_{0},\dots,j_{k}\}$ and
$\Delta_{i_{0},\dots,i_{k}}^{k}\subset X$ is a $k$-dimensional simplex
with vertices $\{i_{0},\dots,i_{k}\}$. 
\end{enumerate}
In particular, the image of $f$ is $N(\Delta_{i_{0},\dots,i_{k}}^{k})\subset N(X)$.\end{prop}
\begin{rem*}
It follows that $f$ is a surjection onto a sub-Steenrod coalgebra
of $N(X)$ generated by a \emph{single simplex} of $X$.\end{rem*}
\begin{proof}
We prove this by induction on $k$. It is clearly true for $k=0$.

In dimension $1$, $\mathbf{f}|\{j_{0},j_{1}\}$ is bijective so that
the image of a $1$-simplex defined by $(i<j)$ with distinct endpoints
must be a $1$-simplex $(k<\ell)$ in $X$ with distinct endpoints.
Since the boundary is $j_{1}-j_{0}$ in $N(\Delta_{j_{0},j_{1}}^{1})$
and since $f$ is a \emph{chain-map,} it follows that $i<j\in\mathbf{n}$
implies that $\mathbf{f}(i)\le\mathbf{f}(j)\in\mathbf{m}$ (compare
with the 1-dimensional case in the proof of corollary~\ref{cor:automorphisms-trivial}).

If $\{j_{0},\dots,j_{k}\}\subset\mathbf{n}$ maps to $\{i_{0},\dots,i_{k}\}$
under $\mathbf{f}$, then the inductive hypothesis implies that every
$k$-element subset of $\{j_{0},\dots,j_{k}\}$ represents a $k-1$-dimensional
face of $\Delta^{n}$ that maps to a $k-1$-dimensional simplex of
$X$. In other words the faces of $\Delta_{j_{0},\dots,j_{k}}^{k}$
map nontrivially to $k-1$-dimensional simplices of $X$ that are
\emph{distinct} because
\[
\mathbf{f}|\{j_{0},\dots,j_{k}\}
\]
 is bijective.

If $\sigma\in N(\Delta_{j_{0},\dots,j_{k}}^{k})_{k}$ represents the
sub-simplex $\Delta_{j_{0},\dots,j_{k}}^{k}\subset\Delta^{n}$, corollary\ref{cor:n-simplices-map-to-simplices}
implies that $f(\sigma)$ is either $0$ or a generator of $N(X)$
representing a $k$-simplex, $\Delta_{i_{0},\dots,i_{k}}^{k}$. The
case where $f(\sigma)=0$ contradicts the fact that $f$ is a chain-map
and its \emph{boundary} maps to a nonzero element of $N(X)$ (a linear
combination of \emph{distinct} $k-1$-dimensional simplices of $X$).
\end{proof}
We finally arrive at:
\begin{thm}
\label{thm:Steenrod-morphisms-geometric}If $X$ is an ordered simplicial
complex, any Steenrod-coalgebra morphism
\[
f:N(\Delta^{n})\to N(X)
\]
is induced by a simplicial map
\[
\hat{f}:\Delta^{n}\to X
\]
\end{thm}
\begin{proof}
Proposition~\ref{prop:image-steenrod-is-simplex} implies that the
image of $f$ is generated by a single simplex of $X$, and corollary~\ref{cor:automorphisms-trivial}
and proposition~\ref{prop:degeneracies} imply that this is induced
by a simplicial map.
\end{proof}
We define a complement to the $\cfn *$-functor: 
\begin{defn}
\label{def:fc}Define a functor
\[
\ss\nfc *:\steen\to\ss
\]
to the category of simplicial sets, as follows:

If $C\in\steen$, define the $n$-simplices of $\ss\nfc C$ to be
the Steenrod coalgebra morphisms
\[
\ns n\to C
\]
where $\ns n=\cfn{\Delta^{n}}$ is the normalized chain-complex of
the standard $n$-simplex, equipped with the Steenrod coalgebra structure
defined in theorem~B.2 of \cite{smith-steenrod}.

Face-operations are duals of coface-operations
\[
d_{i}:[0,\dots,i-1,i+1,\dots n]\to[0,\dots,n]
\]
with $i=0,\dots,n$ and vertex $i$ in the target is \emph{not} in
the image of $d_{i}$.

Degeneracies are duals of codegeneracy operators
\begin{align*}
s_{i}:[0,\dots,i,i+1,\dots,n+1] & \to[0,\dots,n]\\
i & \mapsto i\\
i+1 & \mapsto i
\end{align*}
\end{defn}
\begin{prop}
\label{prop:ux-map}If $X$ is an ordered simplicial complex there
exists a natural inclusion
\[
u_{X}:X\to\ss\nfc{\cfn X}
\]
\end{prop}
\begin{proof}
To prove the first statement, note that any simplex $\Delta^{k}$
in $X$ comes equipped with a canonical inclusion
\[
\iota:\Delta^{k}\to X
\]
The corresponding order-preserving map of vertices induces an Steenrod-coalgebra
morphism 
\[
\cfn{\iota}:\cfn{\Delta^{k}}=\ns k\to\cfn X
\]
so $u_{X}$ is defined by
\[
\Delta^{k}\mapsto\cfn{\iota}
\]
It is not hard to see that this operation respects face-operations.
\end{proof}
So, $\ss\nfc{\cfn X}$ naturally contains a copy of $X$. The interesting
question is how much \emph{more} it contains:
\begin{thm}
\label{thm:simplicial-complexes-determined}If $X\in\simpc$ is an
ordered simplicial complex, then 
\[
\ss\nfc{\cfn X}=\ds(X)
\]
 and the canonical map
\[
\iota_{X}:X\to\ss\nfc{\cfn X}
\]
that sends $X$ to the nondegenerate simplices of $\ds(X)$ induces
a homeomorphism 
\[
H:|X|\to|\ds(X)|
\]
of topological realizations.\end{thm}
\begin{rem*}
Since $\ds(X)$ is degeneracy-free, its nondegenerate simplices form
a sub-complex. The homeomorphism, $H$, is essentially the identity
map.\end{rem*}
\begin{proof}
Theorem~\ref{thm:Steenrod-morphisms-geometric} implies that
\[
\ss\nfc{N(X)}=\bigsqcup_{\mathbf{m}\twoheadrightarrow\mathbf{n}}X_{n}=\ds(X)
\]
since Steenrod-coalgebra morphisms between Steenrod coalgebras of
simplices are \emph{always} induced by \emph{simplicial maps.} The
vertex maps are monomorphisms for the simplices of $X$ \emph{only,}
and \emph{proper} surjections (i.e. not 1-1) \emph{only} for the added
degenerate simplices. The added degenerate simplices are only subject
to the basic identities between face- and degeneracy-operators. The
conclusion follows from proposition~\ref{prop:homotopy-equiv-sd-ds}. 
\end{proof}
This immediately implies
\begin{cor}
\label{cor:cellular-determines-pi1}If $X$ and $Y$ are ordered simplicial
complexes, any morphism of Steenrod coalgebras 
\[
g:\cfn X\to\cfn Y
\]
induces a map
\[
\hat{g}:\ss\nfc{\cfn X}=\ds(X)\to\ss\nfc{\cfn Y}=\ds(Y)
\]
of simplicial sets and a map of topological realizations
\[
|X|=|\ds(X)|\xrightarrow{\hat{|g|}}|\ds(Y)|=|Y|
\]
where $|*|$ is topological realization. In addition, the diagram
\[
\xymatrix{{H_{i}(N(X))}\ar[r]^{g}\ar[d]_{H_{i}(j_{X})}^{\cong} & {H_{i}(N(Y))}\ar[d]_{\cong}^{H_{i}(j_{Y})}\\
{H_{i}(C(\ds(X)))}\ar[r]_{H_{*}(\hat{g})} & {H_{i}(C(\ds(Y)))}
}
\]
commutes for all $i\ge0$, where $j_{X}$ and $j_{Y}$ are chain-maps
induced by the inclusion of $N(X)$ and $N(Y)$ into $C(\ds(X))$
and $C(\ds(Y))$, respectively.

If $g$ is an isomorphism, then $|\hat{g}|$ is a homeomorphism.\end{cor}
\begin{proof}
A morphism $g:\cfn X\to\cfn Y$ induces a morphism of simplicial sets
\[
\ss\nfc g:\ss\nfc{\cfn X}\to\ss\nfc{\cfn Y}
\]
which is an isomorphism (and homeomorphism of topological realizations)
if $g$ is an isomorphism. The conclusion follows from theorem~\ref{thm:simplicial-complexes-determined}.
The chain-maps $j_{X}$ and $j_{Y}$ are nothing but the inclusions
of the sub-chain-complexes generated by \emph{nondegenerate} simplices
--- which are well-known to be homology equivalences.
\end{proof}
\appendix

\section{Delta complexes and simplicial sets\label{sec:Delta-complexes-and}}

We consider variations on the concept of simplicial set.
\begin{defn}
\label{def:delta-complexes}Let $\ords$ be the ordinal number category
whose morphisms are order-preserving monomorphisms between them. The
objects of $\ords$ are elements $\mathbf{n}=\{0\to1\to\cdots\to n\}$
and a morphism 
\[
\theta:\mathbf{m}\to\mathbf{n}
\]
 is a strict order-preserving map ($i<k\implies\theta(i)<\theta(j)$).
Then the category of \emph{delta-complexes,} $\dcat$, has objects
that are contravariant functors
\[
\ords\to\mathbf{Set}
\]
to the category of sets. The chain complex of a delta-complex, $X$,
will be denoted $N(X)$.\end{defn}
\begin{rem*}
In other words, delta-complexes are just simplicial sets \emph{without
degeneracies. }Note that ordered simplicial complexes are particular
types of delta-complexes.

A simplicial set gives rise to a delta-complex by ``forgetting''
its degeneracy-operators --- ``promoting'' its degenerate simplices
to nondegenerate status. Conversely, a delta-complex can be converted
into a simplicial set by equipping it with degenerate simplices in
a mechanical fashion. These operations define functors:\end{rem*}
\begin{defn}
\label{def:sd-ds-functors}The functor
\[
\sd:\ss\to\dcat
\]
is defined to simply drop degeneracy operators (degenerate simplices
become nondegenerate). The functor
\[
\ds:\dcat\to\ss
\]
equips a delta complex, $X$, with degenerate simplices and operators
via
\begin{equation}
\ds(X)_{m}=\bigsqcup_{\mathbf{m}\twoheadrightarrow\mathbf{n}}X_{n}\label{eq:ds-functor}
\end{equation}
for all $m>n\ge0$.\end{defn}
\begin{rem*}
The functors $\sd$ and $\ds$ were denoted $F$ and $G$, respectively,
in \cite{rourke-sanderson-delta-complex}. Equation~\ref{eq:ds-functor}
simply states that we add all possible degeneracies of simplices in
$X$ subject \emph{only} to the basic identities that face- and degeneracy-operators
must satisfy. 

Although $\sd$ promotes degenerate simplices to nondegenerate ones,
these new nondegenerate simplices can be collapsed without changing
the homotopy type of the complex: although the degeneracy operators
are no longer built in to the delta-complex, they still define contracting
homotopies.
\end{rem*}
The definition immediately implies that
\begin{prop}
\label{prop:cx-is-nfx}If $X$ is a simplicial set and $Y$ is a delta-complex,
$C(X)=N(\sd(X))$, $N(\ds(Y))=N(Y)$, and $C(X)=N(\ds\circ\sd(X))$.\end{prop}
\begin{defn}
\label{def:degeneracy-free}A simplicial set, $X$, is defined to
be \emph{degeneracy-free} if 
\[
X=\ds(Y)
\]
 for some delta-complex, $Y$.\end{defn}
\begin{rem*}
Compare definition~1.10 in chapter VII of \cite{goerss-jardine}\footnote{ Their definition has a typo, stating that $\ords$ consists of \emph{surjections}
rather than \emph{injections}.}). In a manner of speaking, $X$ is freely generated by the degeneracy
operators acting on a basis consisting of the simplices of $Y$. Lemma~1.2
in chapter~VII of \cite{goerss-jardine} describes other properties
of degeneracy-free simplicial sets (hence of the functor $\ds$).
\end{rem*}
In \cite{rourke-sanderson-delta-complex}, Rourke and Sanderson also
showed that one could give a ``somewhat more intrinsic'' definition
of degeneracy-freeness:
\begin{prop}
\label{prop:intrinsic-degenracy-free}If $X$ is a simplicial set,
let $\mathrm{Core}(X)$ consist of the nondegenerate simplices and
their faces. This is a delta-complex and there exists a canonical
map
\[
c:\ds(\mathrm{Core}(X))\to X
\]
sending simplices of $\mathrm{Core}(X)$ to themselves in $X$ and
degeneracies to suitable degeneracies of them. Then $X$ is degeneracy-free
if and only if $c$ is an isomorphism.
\end{prop}
Theorem~1.7 of \cite{rourke-sanderson-delta-complex} shows that
there exists an adjunction:

\begin{equation}
\ds:\dcat\leftrightarrow\ss:\sd\label{eq:ds-sd-adjunction}
\end{equation}
The composite (the \emph{counit} of the adjunction)
\[
\sd\circ\ds:\dcat\to\dcat
\]
maps a delta complex into a much larger one --- that has an infinite
number of (degenerate) simplices added to it. There is a natural inclusion
\begin{equation}
\iota_{X}:X\to\sd\circ\ds(X)\label{eq:unit-of-adjunction}
\end{equation}
 and a natural map (the \emph{unit} of the adjunction)
\begin{equation}
g:\ds\circ\sd(X)\to X\label{eq:ds-sd-unit}
\end{equation}
The functor $g$ sends degenerate simplices of $X$ that had been
``promoted to nondegenerate status'' by $\sd$ to their degenerate
originals --- and the extra degenerates added by $\ds$ to suitable
degeneracies of the simplices of $X$. 

Rourke and Sanderson also prove: 
\begin{prop}
\label{prop:homotopy-equiv-sd-ds}If $X$ is a simplicial set and
$Y$ is a delta-complex then
\begin{enumerate}
\item $|Y|$ and $|\ds Y|$ are homeomorphic,
\item the map $|g|:|\ds\circ\sd(X)|\to|X|$ is a homotopy equivalence, so
that $|\iota_{Y}|:|Y|\to|\sd\circ\ds(Y)|$ is a homotopy equivalence,
\item $\sd:H\ss\to H\dcat$ defines an equivalence of categories, where
$H\ss$ and $H\dcat$ are the homotopy categories, respectively, of
$\ss$ and $\dcat$. The inverse is $\ds:H\dcat\to H\ss$. 
\end{enumerate}
\end{prop}
\begin{rem*}
Here, $|*|$ denotes the topological realization functors for $\ss$
and $\dcat$.\end{rem*}
\begin{proof}
The first two statements are proposition~2.1 of \cite{rourke-sanderson-delta-complex}. 
\end{proof}
\bibliographystyle{amsplain}

\providecommand{\bysame}{\leavevmode\hbox to3em{\hrulefill}\thinspace}
\providecommand{\MR}{\relax\ifhmode\unskip\space\fi MR }
\providecommand{\MRhref}[2]{%
  \href{http://www.ams.org/mathscinet-getitem?mr=#1}{#2}
}
\providecommand{\href}[2]{#2}


    \end{document}